\documentclass[12pt]{article}
\usepackage{amsxtra,amssymb,amsfonts,amsthm,amsmath}

\newcommand{\baar}[1]{\overline{#1}}
\newcommand{\beq}{\begin{equation}}
\newcommand{\eeq}{\end{equation}}

\def\elabel#1{\label{e:#1}}
\def\eq#1/{(\ref{e:#1})}
\def\Section#1/{Section~\ref{s:#1}}
\def\Table#1/{Table~\ref{t:#1}}
\def\Figure#1/{Figure~\ref{f:#1}}

\makeatletter
\@addtoreset{equation}{section}
\makeatother

\input epsf

\theoremstyle{plain}
\newtheorem{theorem}{Theorem}
\newtheorem*{theorem*}{Theorem}
\newtheorem{lemma}{Lemma}

\newcommand{\refT}[1]{Theorem~\ref{T:#1}}
\newcommand{\refS}[1]{Section~\ref{#1}}
\newcommand{\refL}[1]{Lemma~\ref{L:#1}}

\begin{document}

\title{ { {\bf Binomial ARMA Count Series\\ from Renewal Processes} }}

\author{
{\sc Sergiy Koshkin  \&~ Yunwei Cui~ } \\
Computer and Mathematical Sciences Department\\
University of Houston Downtown\\
Houston, TX 77002, USA\\
koshkins@uhd.edu; cuiy@uhd.edu 
}

\date{}

\maketitle

\begin{abstract} This paper describes a new method for generating stationary integer-valued time series 
from renewal processes. We prove that if the lifetime distribution of renewal processes is nonlattice  
and the probability generating function is rational, then the generated time series satisfy causal 
and invertible ARMA type stochastic difference equations. The result provides 
an easy method for generating integer-valued time series with ARMA type autocovariance functions.
Examples of generating binomial ARMA$(p,p-1)$ series 
from lifetime distributions with constant hazard rates after lag $p$ are given as an illustration. 
An estimation method is developed for the AR($p$) cases. 
\end{abstract}

\vspace{.12in} 

\vspace{.12in} 
\noindent MSC primary 37M10, secondary 62M10

\section{Introduction}
Integer-valued time series have a broad range of applications including demographic studies, business planning and risk management. Among models developed for them integer-valued autoregressive (INAR) ones appear most frequently in the literature, see McKenzie (2003) for a review. However, their applicability is limited by their autocorrellation functions always being non-negative. More recent approaches include random coefficient processes of Zhang et al. (2007), 
applications of the rounding operator of Kachour and Yao (2009), and the $p$'th order random coefficient autoregressive 
process of Wang and Zhang (2011). 

We pursue a different method of generating time series by superposing independent integer-valued renewal processes, 
which unlike INAR models can induce negative autocorrelation functions. The method was originally proposed by Blight (1989)
and developed by Cui and Lund (2009) to generate a variety of time series, Markov and long memory, with binomial and other marginals. Following Cui and Lund (2009) we choose renewal processes to be stationary from the very beginning to make the generated count process stationary. As Blight noticed, its autocovariance generating function can be easily expressed in terms of the lifetime distribution. In a couple of examples he computed it had the structure of the autocovariance of an autoregressive moving average (ARMA) count series, and he seemd to beleive this to be the case whenever the generating function is rational. The question reduces to a non-trivial factorization of the numerator 
of the generating function, which Blight performed explicitly in his examples. The main purpose of this paper is to prove that the resulting count series is always ARMA if a lifetime distribution is nonlattice and has rational probability generating function, see \refT{ARMAfac}. Our proof involves palindromic polynomials and some subtle properties of probability characteristic functions. As an illustration, we use lifetime distributions with constant hazard rates after lag $p$ to generate binomial ARMA($p,p-1$) count series and study their properties. For $p>2$ explicit ARMA factorization is not available.

The paper is organized as follows. \refS{s2} recalls the construction of renewal count processes.
In \refS{s3} we review the definition of ARMA processes and state our main result, \refT{ARMAfac}, on generation of integer-valued binomial ARMA time series. The proof is given in \refS{s4} and in \refS{s5} we apply our theorem to generate binomial ARMA$(p,p-1)$ time series from renewal processes with constant hazard rates after lag $p$, and show that the former possess the $p$'th order Markov property. Finally, we draw some conclusions.

\section{Renewal count processes}\label{s2}

This section gives a brief review of renewal processes, see Feller (1968) and Ross (1995) 
for a thorough treatment. Let $L$ be a nonnegative random variable, 
called lifetime, taking values in $\{ 1, 2, \ldots \}$
with $P(L=n)=f_n$ and $0<f_1<1$. Let $L_0, L_1, L_2, \ldots$ be 
independent nonnegative integer-valued random variables
with $L_1, L_2, \ldots$ having the same distribution as $L$. 
We allow $L_0$ to have a distribution other than $L$. Then a renewal is said to happen at time $n$ if $L_0+L_1+\cdots+L_k=n$ for some $k \geq 0$. If $L_0$ has unit mass at $0$, i.e. $L_0\equiv0$, 
the process is called non-delayed or pure, otherwise it is called delayed. 

For a non-delayed process let $u_n$ be the probability that a 
renewal occurs at time $n$, then $u_n$ satisfies
$u_0=1$ and $u_n=\Sigma_{j=0}^{n-1}u_jf_{n-j}$, $n\geq 1$. For a delayed process let $\nu_n$ be the 
probability of a renewal at time $n$, then $\nu_0=b_0$, $\nu_n=\sum^n_{k=0}b_k u_{n-k}$ for $n \geq 1$, where $b_n=P(L_0=n)$. When $L$ is nonlattice, has finite mean, and $b_n=\mu^{-1}P(L>n)$, i.e. $L_0$ has the 
so-called equilibrium or first derived distribution of $L$, 
the delayed process is stationary with $\nu_n \equiv\mu^{-1}$ (Ross, 1995). 

For a stationary renewal process define the following sequence of Bernoulli random variables:  
$X_t=1$ if a renewal occurs at time $t$, otherwise $X_t=0$. 
It can be shown that $X_t$ is strictly stationary with 
\[
\gamma(h)=\mbox{cov}(X_t, X_{t+h})=\frac{1}{\mu}(u_h-\frac{1}{\mu}).
\]

Many types of integer-valued time series with different marginal distributions 
can be generated by the above model.
If we superposition (Cox and Smith, 1954) $M$ independent and identical 
Bernoulli sequences $X_{i,t}$, $i=1,2,\ldots,M$
and define $Y_t=\sum_{i=1}^M X_{i,t}$ for $t \geq 0$, 
then $Y_t$ is strictly stationary with binomial marginal distribution. The
autocovariance of $Y_t$ is
\[
\mbox{cov}(Y_t, Y_{t+h})= 
\frac{M}{\mu} \left( u_h-\frac{1}{\mu} \right).
\]
If $L$ has a constant hazard rate after lag $1$ then ${Y_t}$ is Markov. 
Long memory binomial series can also be generated by taking $L$ with finite 
mean but an infinite second moment (see Cui and Lund, 2009, for details).

\section{ARMA processes}\label{s3}

A stationary process $X_t$ is called ARMA$(p,q)$ process if for every $t$
\[
X_t-\phi_1X_{t-1}-\cdots-\phi_p X_{t-p}=Z_t+\theta_1Z_{t-1}+\theta_{2}Z_{t-2}+\cdots+\theta_{q}Z_{t-q},
\]

\noindent where $Z_t$ is a white noise process with variance $\sigma^2$. 
It is convenient to describe ARMA$(p,q)$ processes using autocovariance generating functions.
In general, if $\gamma(h)$ is the autocovariance function of a stationary process
then its autocovariance generating function is defined by  
\[
G(z)=\sum_{h=-\infty}^{\infty}\gamma(h)z^h.
\]

\noindent For an ARMA$(p, q)$ process, the classic result shows that
\beq
G(z)=\sigma^2\frac{\theta(z)\theta(z^{-1})}{\phi(z)\phi(z^{-1})},
\elabel{gene}
\eeq

\noindent where $\phi(z)=1-\phi_1z-\phi_2z^2-\cdots-\phi_pz^p$ and $\theta(z)=1+\theta_1z+\theta_2z^2+\cdots+\theta_qz^q$ 
are called the autoregressive characteristic polynomial and  the moving average characteristic polynomial respectively. 
It can be shown that a stationary process is ARMA$(p, q)$ if its autocovariance generating function can 
be written in the form \eq gene/, where both $\phi(z)$ and $\theta(z)$ have all their roots outside the 
unit circle (see Priestley, 1981).

Now let $Y_t$ be the integer-valued time series with binomial marginal distributions defined 
in the last section. The probability generating function of lifetime $L$ is defined to be
\[
F(z):=\sum_{n=1}^\infty f_nz^n.
\] 

\noindent As shown by Blight (1989), the autocovariance generating function of $Y_t$ is given by 
\[	
G(z)=\frac{M}{\mu}\,\frac{1-F(z)F(\frac{1}{z})}{[1-F(z)][1-F(\frac{1}{z})]},
\elabel{autogen}
\]
where $M$ is the number of independent and identical renewal processes and $\mu$ is the mean of 
$L$. If $F(z)$ is rational, i.e. $F(z)=P(z)/Q(z)$ with $P(z)$ and $Q(z)$ polynomials, then
\beq
G(z)=\frac{M}{\mu}\,\frac{Q(z)Q(\frac{1}{z})-P(z)P(\frac{1}{z})}{[Q(z)-P(z)][Q(\frac{1}{z})-P(\frac{1}{z})]}.
\elabel{genandgen}
\eeq

Recall that a discrete probability distribution $P(L=n)=f_n$, 
$n\in Z$ is called lattice if it is supported on a 
sublattice of integers, i.e. there exists a $d>0$ such that 
$\sum_{k=0}^{\infty}P(L=kd)=1$. We will show that if $L$ is nonlattice 
and has a rational probability generating function, 
then \eq genandgen/ can always be factorized as in \eq gene/. More precisely, the following is true.

\begin{theorem}\label{T:ARMAfac} Let $L$ be a nonlattice distribution 
with a rational probability generating function $F(z)=P(z)/Q(z)$, written in lowest terms, and variance $\sigma_{L}^2$. 
Then it represents a causal and invertible ARMA process. Moreover, its autocovariance generating function 
can be factorized as $G(z)=\frac{kM}{\mu}\,\frac{\theta(z)\theta(\frac1z)}{\phi(z)\phi(\frac1z)}$ with 
$k=\,\frac{\sigma_{L}^2\,Q^2(1)}{\theta^{\,2}(1)\,Q^2(0)}$, where $\phi(z)$ and $\theta(z)$
have all their zeros outside the unit circle, and no common zeros.  
\end{theorem}
\noindent  Formula for $k$ given in Blight (1989) has a missing factor. We prove \refT{ARMAfac} in the next section. 

\section{ARMA factorization}\label{s4}

In this section we prove our main result, \refT{ARMAfac}.
First, recall a result on nonlattice distributions, which 
is crucial to factorizing \eq genandgen/. Substituting $z=e^{it}$ into the probability generating function 
$F(z)$ we get exactly the characteristic function $\chi(t)=F(e^{it})$ of the lifetime distribution $L$. 
Of course, any characteristic function has $\chi(0)=1$, which corresponds to $F(1)=1$.  
But it turns out that for nonlattice distributions $|\chi(t)|\neq 1$ on $(0,2\pi)$.
In other words, for nonlattice lifetime distributions $F(z)\neq1$ on the 
unit circle except at $z=1$. The following Lemma also shows that in equation \eq genandgen/ $Q(z)-P(z)$ and  
$Q(z)Q(1/z)-P(z)P(1/z)$ have no common zeros on the unit circle except 
at $z=1$.
\begin{lemma}\label{L:zeros} 
Let $f_n$, $n\in \mathrm{Z}$ be a nonlattice distribution and 
$F(z)$ be its probability generating function. Assume that $F(z)$ is rational and 
$F(z)=P(z)/Q(z)$ in lowest terms, i.e. $P(z)$ and $Q(z)$ are polynomials with no common factors.
Then $1-F(z)$ and $1-F(z)F(1/z)$ have only one zero on the unit circle, namely $z=1$, 
and all other zeros are outside the unit circle. Moreover, $z=1$ is the only common zero of $1-F(z)$ and $1-F(z)F(1/z)$, as well as of $Q(z)-P(z)$ and $Q(z)Q(1/z)-P(z)P(1/z)$.  
\end{lemma}
\begin{proof} 
It is proved in Gnedenko and Kolmogorov (1968) that 
$|\chi(t)|<1$ on $(0,2\pi)$ except when $t=0$ if $f_n$ is nonlattice.
This means that $|F(z)|<1$ for $|z|=1$ and $z\neq1$. 
Hence, on the unit circle if $z\neq1$, then $|1-F(z)|\geq|1-|F(z)||>0$, 
and 
$1-F(z)F(1/z)=1-|\chi(t)|^2>0$. 
Consequently, $1-F(z)$ and $1-F(z)F(1/z)$ have no zeros on the unit circle 
except at $z=1$.
Since $F(z)F(1/z)=|P(z)|^2/|Q(z)|^2$,
we see that $Q(z)Q(1/z)-P(z)P(1/z)>0$ on the unit circle for $z\neq1$, 
which means $Q(z)Q(1/z)-P(z)P(1/z)$ also has only $z=1$ as a zero on the unit circle. 

By the maximum modulus principle from complex analysis, $|F(z)|<1$ for all $|z|<1$. Thus
$|1-F(z)|\geq|1-|F(z)||>0$ for all $|z|<1$. We conclude that 
except for $z=1$ all zeros of $1-F(z)$ are outside the unit circle. 
Suppose $z^{*}$ is a common zero of $1-F(z)$ and $1-F(z)F(1/z)$ and $z^*\neq 1$. 
Then $F(z^*)=1$ and $F(1/z^*)=1$. By the above, $z^{*}$ cannot be on the unit circle
so $z^{*}$ or $1/z^{*}$ is inside of it. But this contradicts 
$|F(z)|<1$ for $|z|<1$.

Since $P(z)$ and $Q(z)$ have no common factors $1-F(z)$ and $Q(z)-P(z)$ have the same zeros. 
From the above we conclude that $Q(z)-P(z)$ have all zeros 
outside the unit circle except for $z=1$. Analogously, $Q(z)Q(1/z)-P(z)P(1/z)$ and $1-F(z)F(1/z)$ have the same zeros. We conclude that $Q(z)-P(z)$ and $Q(z)Q(1/z)-P(z)P(1/z)$ 
have no common zeros except for $z=1$. 
\end{proof}

Next we investigate the behavior of $1-F(z)$ and $1-F(z)F(1/z)$ near their common zero $z=1$.
Applying Taylor series expansion to $F(z)$ around $z=1$ one gets $F(z)=1+a(z-1)+b(z-1)^2+o((z-1)^2)$. 
Also, expanding $1/z$ around $z=1$ we have
$$
\frac1z=\frac1{1+(z-1)}=1-(z-1)+(z-1)^2+o((z-1)^2).
$$
Since $z=1$ is a fixed point of $1/z$ we can compose the Taylor expansions:
\begin{multline*}
F\left(\frac1z\right)=1+a(\frac1z-1)+b(\frac1z-1)^2+o\left((\frac1z-1)^2\right)\\
=1+a\left[-(z-1)+(z-1)^2\right]+b(z-1)^2+o((z-1)^2)\\
=1-a(z-1)+(a+b)(z-1)^2+o((z-1)^2).
\end{multline*}
This yields $F(z)F\left(\frac1z\right)=1+(a+2b-a^2)(z-1)^2+o((z-1)^2).$
Thus, $1-F(z)F(1/z)$ has a double zero at $z=1$ unless $a+2b-a^2=0$. 
But $a=F'(1)=E[L]$ is the first moment of lifetime, and $2b=F''(1)=E[L^2]-E[L]$. 
Therefore, $Var[L]=a+2b-a^2$ is the variance of $L$. 
For notation, let $\sigma_{L}^2=a+2b-a^2$, then it is easy to verify that
\beq
F(z)F\left(\frac1z\right)=1+\sigma_{L}^2(z-1)^2+o((z-1)^2).
\elabel{FFLaur}
\eeq

We now factorize equation \eq genandgen/ in the form \eq gene/. 
Recall that we assume $F(z)=P(z)/Q(z)$ in lowest terms. Since $F(1)=1$ the difference 
$Q(z)-P(z)$ from the denominator of \eq genandgen/ 
has a zero at $z=1$. By \refL{zeros}, $Q(z)-P(z)$ can be factorized as 
\[
Q(z)-P(z)=(1-z)Q(0)\phi(z),
\]
\noindent where the polynomial $\phi(z)$ has all
zeros outside the unit circle. We factored out $Q(0)$ to make 
the constant term of $\phi(z)$  equal to $1$
and $\phi(z)=1-\phi_{1}z-\ldots-\phi_{p}z^p$ for some integer $p$
and constants $\phi_i$. After dividing out common factors  
the denominator of \eq genandgen/ takes the desired form (see \eq gene/):
\beq
\frac{[Q(z)-P(z)][Q(1/z)-P(1/z)]}{(1-z)(1-1/z)Q(0)^2}=\phi(z)\phi(1/z).
\elabel{defac}
\eeq
\noindent It remains to factorize the numerator. Here are two simple but important 
observations concerning $Q(z)Q(1/z)-P(z)P(1/z)$. If $a$ is a zero then $1/a$ is also a zero,
and if $a$ is a complex zero then $\baar{a}$ is also a zero since $P(z)$ and $Q(z)$ have real coefficients.
Therefore, zeros of $Q(z)Q(1/z)-P(z)P(1/z)$ come in quartets unless 
some of $a$, $\baar{a}$, $1/a$, $1/\baar{a}$ coincide. 
The latter occurs in two cases. If $a=\baar{a}$ then $a$ is real and the quartet reduces to a real pair $a$, $1/a$;
if $a$ is complex and on the unit circle the quartet reduces to a complex conjugate 
pair $a$, $\baar{a}$.

In fact, $Q(z)Q(1/z)-P(z)P(1/z)$ is closely related to palindromic polynomials in which 
coefficients read the same from left to right as from right to left. Namely, it becomes 
a palindromic polynomial after being multiplied by the 
highest power of $z$. Zeros of real palindromic polynomials 
also generically come in quartets $a$, $\baar{a}$, $1/a$, and $1/\baar{a}$.

\begin{lemma}\label{L:numfac}
For a nonlattice lifetime distribution with rational generating 
function $F(z)=P(z)/Q(z)$, written in lowest terms, there exist a real polynomial $\theta(z)$ with all zeros outside the 
unit circle, and a constant $c$ such that 
\[
Q(z)Q(1/z)-P(z)P(1/z)=c\,(1-z)(1-1/z) \theta(z)\theta(1/z),
\]
\noindent where $\theta(z)=1+\theta_1z+\theta_2z^2+\cdots+\theta_qz^q$ for 
some integer $q$ and real constants $\theta_i$.
\end{lemma}
\begin{proof}
Since $P(z)$ and $Q(z)$ have no common factors, $1-F(z)F(1/z)$ and $Q(z)Q(1/z)-P(z)P(1/z)$ have identical zeros.
It follows from \eq FFLaur/ that $z=1$ is a double zero of the former and therefore of the latter. In other words, 
$(1-z)(1-1/z)$ can be factored from $Q(z)Q(1/z)-P(z)P(1/z)$. \refL{zeros} tells us that $Q(z)Q(1/z)-P(z)P(1/z)$ has 
no other zeros on the unit circle. Therefore, the remaining factors come in quartets 
$$
(1-\frac{1}{a_j}z),(1-\frac{1}{\baar{a_j}}z),(1-\frac{1}{a_j}\frac{1}{z}),(1-\frac{1}{\baar{a_j}}\frac{1}{z}),
$$
with $a_j$ complex and $|a_j|>1$ or pairs 
$$
(1-\frac{1}{a_k}z)(1-\frac{1}{a_k}\frac{1}{z})
$$
with $a_k$ real and $|a_k|>1$. Define $\theta(z)$ to be the product of all factors $(1-\frac{1}{a_j}z)(1-\frac{1}{\baar{a_j}}z)$ in the first case, and all factors $(1-\frac{1}{a_k}z)$ in the second case.
It is clear that $\theta(z)$ has real coefficients since
$$
(1-\frac{1}{a_j}z)(1-\frac{1}{\baar{a_j}}z)=1-(\frac{1}{a_j}+\frac{1}{\baar{a_j}})z+\frac{1}{|a_j|^2}z^2.
$$
\end{proof}
\noindent Now we are in a position to prove the main theorem. 
\begin{proof}[Proof of {\rm\bf \refT{ARMAfac}}]  
Dividing the numerator and the denominator of equation \eq genandgen/ 
by $(1-z)(1-1/z)Q(0)^2$ we get \eq defac/ as the new denominator.
For the numerator we apply \refL{numfac} to get a real 
polynomial $\theta$ satisfying
\beq
k\,\theta(z)\theta(\frac1z)=\,\frac{Q(z)Q(\frac{1}{z})-P(z)P(\frac{1}{z})}{(1-z)(1-\frac1z)Q^2(0)},
\elabel{thdef}
\eeq
\noindent where $k$ is selected to make $\theta(z)$ have unit constant term. 
To compute $k$ we divide both sides of \eq thdef/ by $Q(z)Q(1/z)$ and get 
$$
k\,\frac{\theta(z)\theta(\frac1z)}{Q(z)Q(\frac{1}{z})}
=\,\frac{1-F(z)F\left(\frac1z\right)}{(1-z)(1-\frac1z)Q^2(0)}
$$
Set $z\to1$ on both sides. The lefthand side becomes simply $k\,\theta^{\,2}(1)/Q^2(1)$. 
The righthand side is seen from \eq FFLaur/ to approach $\sigma_{L}^2/Q^2(0)$. 
Solving for $k$ yields the desired formula. Since $k>0$ our $Y_t$ is an ARMA time series. 
By Lemmas \ref{L:zeros} and \ref{L:numfac}, $\phi(z)$ and $\theta(z)$ have 
all zeros outside the unit circle and no common zeros. It follows that the corresponding 
ARMA process is causal and invertible (Brockwell and Davis, 1991, Ch.3).
\end{proof}

\section{Binomial ARMA$(p,p-1)$ time series}\label{s5}

In this section we show how to generate some binomial ARMA$(p,p-1)$
time series using \refT{ARMAfac}. We use distributions
with constant hazard rates after lag $p$ as lifetimes. We also discuss Markov properties 
of the generated series.

If a lifetime distribution has a constant hazard rate after lag $2$, the probability mass
function is $P(L=n)=f_{3} r^{n-3}$ with $0<f_{3},~ r<1$ for $n\geq 3$. It is clearly nonlattice.
It can also be shown that the hazard rate is $h_k=P(L=k|L\geq k)=(1-r)$ for $k\geq3$.
The probability generating function of $L$ is 
\[
F(z)=\frac{z[f_1+(f_2-f_1r)z+(f_3-f_2r)z^2]}{1-rz}.
\] 
\noindent From the last section we know that
$Q(z)=1-rz$ and $P(z)=z[f_1+(f_2-f_1r)z+(f_3-f_2r)z^2]$. 
Plugging $z=1/r$ into $P(z)$ we get $P(1/r)=f_3/r^2\neq0$. 
Since $1/r$ is the only zero of $Q(z)$ polynomials $Q(z)$ and $P(z)$ have no common factors. 

To factorize the covariance generating function we first compute 
\[
Q(z)-P(z)=1-(r+f_1)z-(f_2-f_1r)z^2-(f_3-f_2r)z^3=(1-z)(1-\phi_1z-\phi_2z^2),
\]
\noindent with $\phi_1=r+f_1-1$, $\phi_2=f_2r-f_3$.
The numerator of \eq genandgen/, $Q(z)Q(z^{-1})-P(z)P(z^{-1})$,
has a factor $(1-z)(1-z^{-1})$. Besides a double zero at $z=1$ there exists another pair of zeros, 
$a_1$ and $a_1^{-1}$. Since $Q(z)Q(z^{-1})-P(z)P(z^{-1})=(1-z)(1-z^{-1})(\pi_0 z +\pi_1+\pi_0z^{-1})$, 
where $\pi_0=f_1(f_3-f_2r)$, $\pi_1=f_1f_2(1-r)^2+f_1f_3(2-r)+r(1-f_1^2-f_2^2)+f_2f_3$, 
one can solve for $a_1$ from $\pi_0 z +\pi_1+\pi_0z^{-1}=0$ and get 
\[
a_1=\frac{-\pi_1-\sqrt{\pi_{1}^2-4\pi_0^2}}{2\pi_0}.
\]
\noindent Letting $\theta=-a^{-1}_1$ one has as in \refL{numfac}
\beq
Q(z)Q(z^{-1})-P(z)P(z^{-1})=k(1-z)(1-\frac{1}{z})(1+\theta z)(1+\theta z^{-1}),
\elabel{bfnum}
\eeq
\noindent where  $k$ can be found from the formula in \refT{ARMAfac}, or 
by comparing the constant terms on both sides of \eq bfnum/. This yields
\[
k=\frac{(1-f_1^2-f_2^2-f_3^2)+r^2(1-f_1^2-f_2^2)+2f_1f_2r+2f_2f_3r}{2+2\theta^2-2\theta}.
\]

\noindent The autocovariance generating function of $Y_t$  is 
\[
G(z)=\frac{M k}{\mu} \frac{(1+\theta z)(1+\theta z^{-1})}{(1-\phi_1z-\phi_2z^2)(1-\phi_1z^{-1}-\phi_2z^{-2})}.
\]
\noindent An AR$(2, 1)$ type stochastic difference equation for $Y_t$ is now readily written. 

More generally, suppose $L$ has a constant hazard rate after lag $p$. Then $L$ has 
$P(L=n)=f_{p+1} r^{n-p-1}$ with $0<f_{p+1},~ r<1$ for $n\geq p+1$. 
The probability generating function of $L$ can be represented by 
a ratio of two polynomials as follows
\begin{eqnarray*}
F(z)&=&f_1 z+f_2 z^2+\ldots+f_{p}z^{p}+\frac{f_{p+1}z^{p+1}}{1-rz}\\\nonumber
&=&\frac{z[f_1+(f_2-f_1r)z+\ldots+(f_{p+1}-f_pr)z^{p}]}{1-rz}.
\end{eqnarray*} 

\noindent As above, we conclude that $P(z)$ and $Q(z)$ have no common factors
since $P(1/r)=f_{p+1}/r^p\neq0$. By \refT{ARMAfac}, $Q(z)-P(z)$ can be factorized as 
$(1-z)(1-\phi_1z-\ldots-\phi_{p}z^p)$ and 
$Q(z^{-1})-P(z^{-1})$ can be factorized as $(1-z^{-1})(1-\phi_1z^{-1}-\ldots-\phi_{p}z^{-p})$
for some real constants $\phi_1,\ldots,\phi_p$. 
Explicit factorization of $Q(z)Q(z^{-1})-P(z)P(z^{-1})$ is no longer possible but \refT{ARMAfac} still ensures
that the stationary time series has ARMA$(p,p-1)$ structure. 

Now we consider the Markov property for our binomial ARMA$(p,p-1)$ processes. For simplicity 
we only treat the case $p=2$, but the proof is analogous, albeit more cumbersome,
for general $p$. The trivariate binomial distribution mentioned below is discussed by
Chandrasekar and Balakrishnan (2002). 
\begin{theorem}\label{T:Markov} Let $Y_t=\sum_{i=1}^MX_{i,t}$, where $X_{i,t}$, $i=\{1, \ldots, M\}$ are
the underlying Bernoulli series. Then $Y_t$ is a second-order Markov chain, i.e. $Y_t$ is independent of 
$\{Y_{t-3}, Y_{t-4},\ldots, Y_0\}$. The vector $(Y_t, Y_{t-1}, Y_{t-2})$ 
has the trivariate binomial distribution with the moment generating function 
\beq
E[e^{Y_ts_1}e^{Y_{t-1}s_2}e^{Y_{t-2}s_3}]=(q+\sum_{1\le i\le 3}p_ie^{s_i}+\sum_{1\le i\le 3}\sum_{1\le j\le 3}p_{ij}e^{s_i}e^{s_j}+p_{123}e^{s_1}e^{s_2}e^{s_3})^{M}.
\elabel{mopgtt-2}
\eeq 
\end{theorem}
\begin{proof}
We start by computing the following probabilities for the underlying Bernoulli series $X_{i,t}$.
\begin{eqnarray*}
p_1&:=&P(X_{i, t}=1, X_{i, t-1}=0, X_{i, t-2}=0)=\mu^{-1}(1-f_1-f_2);\nonumber\\
p_{13}&:=&P(X_{i, t}=1, X_{i, t-1}=0, X_{i, t-2}=1)=\mu^{-1}f_2;\nonumber\\
p_{12}&:=&P(X_{i, t}=1, X_{i, t-1}=1, X_{i, t-2}=0)=\mu^{-1}f_1(1-f_1);\quad 
\parbox{3cm}{~~~~~~~~~~~~~~~~~~~~~~~~~~~~}\nonumber\\
p_{123}&:=&P(X_{i, t}=1,X_{i, t-1}=1, X_{i, t-2}=1)=\mu^{-1}f_1f_1; \nonumber\\
p_3&:=&P(X_{i, t}=0,X_{i, t-1}=0, X_{i, t-2}=1)=\mu^{-1}(1-f_1-f_2);\nonumber\\
p_{23}&:=&P(X_{i, t}=0, X_{i, t-1}=1, X_{i, t-2}=1)=\mu^{-1}f_1(1-f_1);\nonumber\\
p_{2}&:=&P(X_{i, t}=0, X_{i, t-1}=1, X_{i, t-2}=0)=\mu^{-1}(1-f_1)^2;\nonumber\\
q&:=&P(X_{i, t}=0, X_{i, t-1}=0, X_{i, t-2}=0)=1-\sum_{1\le i\le 3}p_i-\sum_{1\le i\le 3}\sum_{1\le j\le 3}p_{ij}-p_{123}.
\end{eqnarray*}
\noindent The conditional probabilities of $X_{i,t}$ can also be explicitly computed. In particular, we use that 
$\mu=1-f_1+\frac{2-r-f_1-f_2}{1-r}$ to simplify $p_{1| 0, 0}$ and get the 
expression for $p_{0| 0, 0}$ from $p_{1| 0, 0}+p_{0| 0, 0}=1$.
\begin{eqnarray}
p_{1| 0, 0}&:=&P(X_{i,t}=1| X_{i,t-1}=0, X_{i,t-2}=0)=1-r;\nonumber\\
p_{1|0, 1}&:=&P(X_{i,t}=1| X_{i,t-1}=0, X_{i,t-2}=1)=f_2/(1-f_1);\quad \parbox{4.5cm}{~~~~~~~~~~~~~~~~~~~~~~~~~~~~}\nonumber\\
p_{1|1, 0}&:=&P(X_{i,t}=1| X_{i,t-1}=1, X_{i,t-2}=0)=f_1;\nonumber\\
p_{1|1,1}&:=&P(X_{i,t}=1| X_{i,t-1}=1, X_{i,t-2}=1)=f_1;\quad \parbox{5.5cm}{~~~~~~~~~~~~~~~~~~~~~~~~~~~~}\nonumber\\
p_{0|0,1}&:=&P(X_{i,t}=0| X_{i,t-1}=0, X_{i,t-2}=1)=(1-f_1-f_2)/(1-f_1);\nonumber\\
p_{0|1,1}&:=&P(X_{i,t}=0| X_{i,t-1}=1, X_{i,t-2}=1)=(1-f_1);\nonumber\\
p_{0|1,0}&:=&P(X_{i,t}=0| X_{i,t-1}=1, X_{i,t-2}=0)=(1-f_1);\nonumber\\
p_{0|0,0}&:=&P(X_{i,t}=0| X_{i,t-1}=0, X_{i,t-2}=0)\nonumber\\
&=&[1-\sum_{1\le i\le 3}p_i-\sum_{1\le i\le 3}\sum_{1\le j\le 3}p_{ij}-p_{123}]/[1-2\mu^{-1}+f_1\mu^{-1}]=r
\elabel{qc}
\end{eqnarray}
\noindent After some algebra one also finds that the probabilities conditioned on $X_{i,t-1}, \ldots, X_{i,0}$ are the same as above, 
\[
P(X_{i,t}| X_{i,t-1}, X_{i,t-2})=P(X_{i,t}| X_{i,t-1}, X_{i,t-2}, X_{i,t-3}\ldots, X_{i,0}),
\]

\noindent i.e. the underlying Bernoulli series $X_{i,t}$ is a second-order Markov chain.

Next, we need to find $P(Y_t|Y_{t-1}, \ldots, Y_0)$. To this end, let $\epsilon_j$, $j=1, \ldots, M$, 
be a zero-one vector with three components, and $\epsilon_j(i)$ denote its $i$'th component, $i=1, \ldots, 3$. 
Define a set of $\epsilon_j$'s by
\[
A_{Y_t|Y_{t-1}, Y_{t-2}}=\left\{\Lambda=(\epsilon_1, \ldots, \epsilon_{M})\left|\quad
\sum_{j=1}^M\epsilon_j(1)=Y_t, \sum_{j=1}^M\epsilon_j(2)=Y_{t-1}, \sum_{j=1}^M\epsilon_j(3)=Y_{t-2}\right.
\right\}.
\]
\noindent By independence and the Markov property of 
the underlying Bernoulli series $X_{i,t}$, we have 
\beq 
P(Y_t|Y_{t-1}, \ldots, Y_0)=\sum_{\Lambda\in A_{Y_t|Y_{t-1}, Y_{t-2}}}\Pi _{j=1}^M p_{\epsilon_j(1)|\epsilon_j(2),\epsilon_j(3)},
\elabel{condi}
\eeq

\noindent where $p_{\epsilon_j(1)|\epsilon_j(2),\epsilon_j(3)}$ can be calculated from 
\eq qc/. 
Since \eq condi/ is not affected by $\{Y_{t-3},\ldots, Y_0\}$, we 
conclude that $P(Y_t|Y_{t-1}, \ldots, Y_0)=P(Y_t|Y_{t-1}, Y_{t-2})$, so  
$\{Y_t\}$ is a second-order Markov chain. The formula for the moment generating function 
$E[e^{Y_ts_1}e^{Y_{t-1}s_2}e^{Y_{t-2}s_3}]$ follows from \eq qc/ by a straightforward 
computation.
\end{proof}

\section{Conclusions}\label{s8}

We proved that the renewal process method generates time series with ARMA type autocovariance 
under fairly broad assumptions. We also gave examples where the generated series have the Markov property. 
As a follow-up, estimation methods for ARMA($p,p-1$) models are worth investigating, for example, conditional least squares and maximum likelihood methods as in Cui and Lund (2009,2010). On a different note, our method can generate periodic count series if one incorporates periodic dynamics into the underlying renewal process. Periodicity is inherent 
in many physical processes, but periodic count series models are scarce in the literature.

\end{document}